\documentclass[11pt]{article}

\usepackage{graphicx}
\usepackage{amssymb}
\usepackage{amsfonts}
\usepackage{amsmath}
\usepackage{hyperref}
\usepackage{tikz}
\usepackage{graphicx,epstopdf}

\usetikzlibrary{arrows}

\newtheorem{theorem}{Theorem}[section]
\newtheorem{proposition}[theorem]{Proposition}
\newtheorem{lemma}[theorem]{Lemma}

\newtheorem{definition}[theorem]{Definition}
\newtheorem{example}[theorem]{Example}

\newenvironment{proof}[1][Proof]{\textbf{#1.} }{\ \rule{0.5em}{0.5em}}

\begin{document}

\author{ Yuri Muranov \\
University  of \\ Warmia and Mazury\\
 Olsztyn, Poland 
\and 
Anna Szczepkowska \\
University of \\ Warmia and Mazury\\
 Olsztyn, Poland
\and  
Vladimir Vershinin\\ 
Universit\'e de Montpellier\\  Montpellier, France
}
\title{Path  homology  of directed hypergraphs}
\maketitle

\begin{abstract}
We describe various  path homology theories  constructed for a directed hypergraph. 
We introduce  the category  of directed hypergraphs and the notion of a homotopy in this category. 
Also, we investigate the functoriality and the homotopy invariance of the introduced path homology groups. 
We provide examples of computation of these homology groups. 
\end{abstract}

\bigskip



\section{Introduction}\setcounter{equation}{0}\label{S1}

Directed hypergraphs are the generalization of digraphs and  have been widely  used  in  
discrete mathematics and computer science, see e.g. \cite{Sur}, \cite{Berge},  \cite{mor},  and \cite{Gallo}. 
In particular,  the  directed hypergraphs  give effective  tools  for the investigation of databases and 
structures on complicated  discrete  objects. 

Recently, the topological properties of digraphs, hypergraphs, multigraphs, and 
quivers have been studied using various (co)homologies theories, consult e.g. \cite{Embed}, 
\cite{Graham}, \cite{Betti}, \cite{Parks}, \cite{Mi3}, \cite{Forum}, \cite{Hyper}, \cite{Pcomplex}. 

 In this paper,  we  construct  several functorial and homotopy invariant homology theories on 
 the category of directed hypergraphs  using the  path homology theory introduced in \cite{Mi4},  \cite{Hyper},   \cite{Mi2}, \cite{Pcomplex}, and 
\cite{Forum}. 

 A rich structure of a directed hypergraph gives a number of opportunities  to define 
 functorially a path complex for   the category of hypergraphs which we construct in the paper. 
 We describe these  constructions in Section \ref{S3}.   We  introduce also a notion of a homotopy in the category of directed hypergraphs  and describe  functorial relations  between homotopy categories of directed hypergraphs,  digraphs,   and path complexes.  
  
The essential difference from the situation of the category of digraphs is 
the  existence of the notion of the density of the path complex that we introduce for the two  of the introduced path complexes of directed hypergraphs. This notion 
gives an opportunity to define a filtration on the corresponding path complex and hence a filtration on
its path homology groups.  We consider all homology groups with  coefficients in a unitary   commutative ring $R$.

In Section  \ref{S2},   we  define a category  of directed hypergraphs  and introduce the notion of homotopy in this category. 

In  Section \ref{S3}, we  construct several path homology theories on the category of directed  hypergraphs.

 \section{Path complexes and homotopy of directed hypergraphs}\label{S2}
\setcounter{equation}{0}

Let $\Pi=(V,P)$ be a path complex with $V=\{0,\dots, n\}$ and $J=\{0,1\}$ be a set, see \cite{Hyper}, \cite{Pcomplex}. 
Define a path complex $\Pi^{\prime}=(V^{\prime}, P^{\prime})$ where $V^{\prime}=\{0^{\prime}, \dots, n^{\prime}\}$ and  
$p^{\prime}=(i_0^{\prime}\dots i_n^{\prime})\in P^{\prime}$ iff $p=(i_0\dots i_n)\in P$. We  identify 
$V\times J=V\times\{0\}\amalg V\times\{1\}$ with  $V\amalg V^{\prime}$. Define  a path complex $\Pi^{\uparrow}=(V\times J, P^{\uparrow})$  where  
\begin{equation*}\label{2.1}
\begin{matrix}
P^{\uparrow}=P\cup P^{\prime}\cup P^{\#}, \\
{P^\#}=\{ q^\#_k=(i_0\dots i_ki_k^{\prime}i^{\prime}_{k+1}\dots i_n^{\prime})|q=(i_0\dots i_ki_{k+1}\dots i_n)\in P \}.\\
\end{matrix}
\end{equation*}
We have   morphisms 
$
i_{\bullet}\colon \Pi \to \Pi^{\uparrow}
$
and $j_{\bullet}\colon \Pi \to \Pi^{\uparrow}$
that are  induced by the natural inclusion $V$ onto $V\times \{0\}$ and  onto   $V\times\{1\}$, respectively.
 
\begin{definition} \label{d2.1}\rm   (i) A  \emph{ hypergraph} is a pair $G=(V,E)$  consisting of  a non-empty   
	set  $V$ and  a  set  $E=\{\bold e_1,\dots, \bold e_n\}$  of    distinct and non-ordered subsets of $V$ such that 
	$
	\bigcup_{i=1}^n \bold  e_i=V
	$ and every $\mathbf e_i$ contains strictly more than one element. 
	The elements of $V$ are called \emph{vertices} and the elements of  $E$ are called \emph{edges}.

	(ii) A \emph{directed hypergraph} $G$  is a pair $(V,E)$ consisting of a  set $V$   and a   set 
	$E=\{\bold e_1,\dots , \bold e_n\}$   where $\bold e_i\in E$    is an ordered pair $(A_i,B_i)$  of disjoint non-empty subsets  of the set $V$  such that 
	$V=\bigcup_{\bold e_i\in  E} (A_i\cup B_i)$. The elements of $V$ are called \emph{vertices} and   the elements of $E$  are called \emph{arrows}.  
	The set $A=\mbox{orig }(A\rightarrow B)$ is called 
	the \emph{origin of the arrow} and the set  $B=\mbox{end}(A\rightarrow B)$
	is called the \emph{end of the arrow}. The elements of $A$ are called 
	the \emph{initial vertices} of $A\to B$ and  the elements of $B$ are called its \emph{terminal vertices}.  
\end{definition}

For a finite set  $X$  let ${\mathbf{P}}(X)$ denote as usual the power set. 
We define  a set 
$
\mathbb P(X)\colon =  \{{\mathbf{P}}(X)\setminus \emptyset\}\times  
\{{\mathbf{P}}(X)\setminus \emptyset\}
$
consisting  of ordered pairs of non-empty subsets of $X$.   
Every  map of finite sets $f\colon V\to W$ induces a map $\mathbb P(f)\colon\mathbb P(V)\to \mathbb P(W)$. 
For a directed hypergraph $G=(V,E)$,   by Definition \ref{d2.1}, we have the natural map  $\varphi_G\colon E\to \mathbb P(V)$ defined by  $\varphi_G(A\to B)\colon =(A,B)$. 

\begin{definition}\label{d2.2} \rm   Let $G=(V_G,E_G)$  and
$H=(V_H,E_H)$  be two  directed hypergraphs. A  \emph{morphism} $f\colon G\to H$  is given by a pair of maps  $f_V\colon V_G\to V_H$ 
and $f_E\colon E_G\to E_H$ 
 such that  the following diagram 
$$
\begin{matrix}
E_G&\overset{\varphi_G}\longrightarrow &\mathbb P(V_G)\\
\ \ \downarrow f_E&&\ \  \downarrow \mathbb P(f_V)\\
E_H&\overset{\varphi_H}\longrightarrow &\mathbb P(V_H)\\
\end{matrix}
$$
is commutative. 
\end{definition}

Let us denote by  $\mathcal{DH}$  the category whose objects are directed  hypergraphs and whose morphisms are morphisms of  directed hypergraphs. 

 For a directed hypergraph $G=(V_G,E_G)$, we  can  consider 
subsets 
${\mathbf{P}}_0(G)\subset  {\mathbf{P}}(V_G)\setminus \emptyset$,  
${\mathbf{P}}_1(G)\subset  {\mathbf{P}}(V_G)\setminus \emptyset$ and 
${\mathbf{P}}_{01}(G)={\mathbf{P}}_{0}(G)\cup {\mathbf{P}}_{1}(G)$
by setting
\begin{equation*}\label{2.1}
\begin{matrix}
{\mathbf{P}}_0(G)=\{A\in  {\mathbf{P}}(V_G)\setminus \emptyset| \exists 
B\in  {\mathbf{P}}(V_G)\setminus \emptyset:  \ A\to B\in E_G\}, \\
{\mathbf{P}}_1(G)=\{B\in  {\mathbf{P}}(V_G)\setminus \emptyset| \exists 
A\in  {\mathbf{P}}(V_G)\setminus \emptyset: \ A\to B\in E_G\}. \\
\end{matrix}
\end{equation*} 

\begin{definition}\label{d2.3}\rm  Let $G=(V_G,E_G)$ and $H=(V_H, E_H)$ be  directed hypergraphs. 
We define the \emph{box product} $G\Box H$ as a directed hypergraph 
with the set of vertices $V_{G\Box H}=V_{G}\times V_{H}$ and the set of  arrows  
$E_{G\Box H}$ consisting of the union of  arrows 
$
\{A\times C\to  B\times C\}
$
with $(A\to B)\in E_G$,   $C\in {\mathbf{P}}_{01}(H)$ and 
$ \{A\times C\to  A\times D\}$
with 
$ (C\to D)\in E_H$,  $A\in 
{\mathbf{P}}_{01}(G)$.
\end{definition}
 Every connected digraph $H=(V_H, E_H)$ can be considered as a  directed  hypergraph 
 with the same set of vertices and  of  a set of arrows of the form $\{v\}\to \{w\}$  whith  $(v\to w)\in E_H$.  
 Hence,  Definition \ref{d2.3} gives naturally 
a box product   $G\Box H$  of a directed hypergraph $G$ and a connected digraph $H$.  
Note that a line digraph $I_n$ defined for example in \cite[Sec. 3.1]{MiHomotopy} is connected and  that we have two digraphs $I_1$, namely $0\to 1$ and $1\to 0$.

\begin{definition}\label{d2.4} \rm i) Two  morphisms  $f_0, f_1\colon G\rightarrow H$ of directed hypergraphs  
	are called \emph{one-step homotopic} if there exists a line digraph $I_1$  and  a morphism  
	$F\colon G\Box I_1\rightarrow H$,  such that 
\begin{equation*}
F|_{G\Box \{0\}}=f_0\colon G\Box \{0\}\rightarrow H,\ \ F|_{G\Box
\{1\}}=f_1\colon G\Box \{1\}\rightarrow H.
\end{equation*}
If the appropriate morphism $F$ called a one-step homotopy exists, we write  $f_0\simeq_1 f_1$.

ii) Two  morphisms  $ f,g\colon G \to  H$  of directed hypergraphs are called  \emph{homotopic},  which we denote   $ f\simeq  g$ if there exists a sequence 
of morphisms 
$
f_i\colon G\to H$ for   $i=0,\dots,  n$  such that 
 $ f=f_0\simeq_1  f_1\simeq_1 \dots \simeq_1  f_n =g$.

iii) Two  directed hypergraphs   $G$ and $H$ are \emph{homotopy equivalent} if there
exist  morphisms 
$
f\colon G\to H$ and $g\colon H\to G$
such that 
$ fg \simeq  { \operatorname{Id}_{H}}$ and 
$ gf\simeq {\operatorname{Id}_G}
$. 
  In such a case,  we write $G\simeq H$ and  call the  morphisms  $f$, $g$ \emph{homotopy inverses } of each other.
\end{definition}

\begin{proposition}\label{p2.5} Two  morphisms  $ f,g\colon G \to  H$  of directed hypergraphs 
	are homotopic if and only if there is a line digraph $I_n$ with $n\geq 0$ and a morphism  
	$F\colon G\Box I_n\rightarrow H$   such that 
$
F|_{G\Box \{0\}}=f_0\colon G\Box \{0\}\rightarrow H,\ \ F|_{G\Box
\{n\}}=g\colon G\Box \{n\}\rightarrow H. \ \  \ 
$  $\blacksquare$
\end{proposition}

The relation "to be homotopic"  is  an equivalence relation on the set of morphisms between  two directed  
hypergraphs, and homotopy equivalence  is an equivalence relation on the set of directed hypergraphs. 
Thus, we can consider  a category 
${h} \mathcal{ DH}$ whose objects are directed hypergraphs and morphisms are the classes of homotopic morphisms. 
We shall call the category ${h} \mathcal{DH}$  by 
\emph{homotopy category of directed hypergraphs}.

\section{Path homology of directed hypergraphs}\setcounter{equation}{0}\label{S3}

\subsection{k-connective  path homology}\label{S31}

 For a  directed hypergraph   $G= (V, E)$ and $c=1,2,3, \dots $
define a path complex  \cite[S3.1]{Pcomplex}  $\mathfrak C^c(G)=(V^c, P_G^c)$   where  
$V^c=V$ and  a path $(i_0\dots i_n)\in P_{V}$ lies in $P_G^c$  iff 
for any pair of consequent vertices $(i_{k}, i_{k+1})$  of the path, we have $i_{k}=i_{k+1}$ or
there are at least $c$  different edges $\bold e_1=(A_1\to B_1), \dots , \bold e_c=(A_c\to B_c)$ such that 
the vertex $i_k$ is the initial vertex and the vertex $i_{k+1}$ is the terminal vertex of  every edge $\bold e_i$. 
The number  $c$ is called the \emph{density} of the path complex $\mathfrak C^c(G)$. It is clear that we have 
a filtration 
\begin{equation}\label{3.1}
\mathfrak C(G)= \mathfrak C^{1}(G)\supset \mathfrak C^{2}(G)\supset \mathfrak C^{3}(G)\supset \dots
\end{equation}

\begin{proposition}\label{p3.1}
 For every morphism of directed hypergraphs $f\colon G\to H$  define a  morphism  
$$
\mathfrak C(f)=(f_V^1, f_p^1)\colon \mathfrak C(G)\to \mathfrak C(H)
$$ 
of path complexes putting $f_V^1\colon =f_V$ and    
$f_p^1\colon=f_p|_{P^1_G}  \colon P^1_G\to  P^1_H$ where     $f_p$ is defined 
by  $f_p(i_0\dots i_n)= \left(f(i_0)\dots f(i_n)\right)$. Then 
 we have the functor $\mathfrak C$  from the category $\mathcal{DH}$ of directed hypergraphs to the category  $\mathcal P$ of path complexes.
 \quad $ \blacksquare$
\end{proposition}

 The functor   $\mathfrak C$ 
 provides the functorial  path homology theory on the category  $\mathcal{DH}$
of directed hypergraphs. For any  directed hypergraph $G$ and $k\in \mathbb N$, we set
 $
 H_*^{\bold {c}( k)}(G)\colon = H_*(\mathfrak C^k(G))$  as \emph{regular path homology groups} 
of  path complex
$\mathfrak C^k(G)$, see  \cite[S2]{Hyper}. We denote $H_*^{\bold c}(G)\colon =H_n^{\bold {c}(1)}(G)$.  

We  call these homology groups  \emph{the connective path homology groups} 
and for $k\geq 2$ \emph{the $k$-connective path homology groups} of the
 directed  hypergraph $G$, respectively.  The connective path homology theory is
 functorial  by Proposition \ref{p3.1}.  However
the $k$-connective homology theory $H_n^{\bold{c}( k)}(G)$ is not functorial for 
$k\geq 2$ as it follows from Example \ref{e3.2} below.  For any directed hypergraph
 $G$  the filtration in
 (\ref{3.1}) induces homomorphisms 
\begin{equation*}
H_n^{\bold c}(G)=H_n^{\bold {c}( 1)}(G)\longleftarrow  H_n^{\bold{c}( 2)}G)\longleftarrow H_n^{\bold{c}( 3)}(G)\longleftarrow \dots \ \ .
\end{equation*}

 Let $\mathcal D$  be a category of  digraphs without loops 
\cite[S2]{MiHomotopy}. A category $\mathcal G$ of graphs  is defined similarly 
\cite[S6]{MiHomotopy}. 

Let $G=(V, E)$ be a directed hypergraph. Define a digraph 
$\mathfrak G(G)=\left(V^d_G, E^d_G\right)$ where $V_G^d=V$ and   an arrow $v\to w$ lies in $E^d_G$ iff 
there is a hyperedge  $(A\to B)\in E$ such that $v\in A, w\in B$. 

\begin{example}\label{e3.2} \rm   i) Let $G=(V,E)$ be a directed hypergraph  such that 
$V$ is the union $A\cup B$ of two  non-empty sets  with empty intersection and  the set $E$ consists 
of one element $\bold e=(A\to B)$. Then $\mathfrak G(G)$ is a  complete bipartite
 digraph with arrows  
from  vertices lying  in $A$ to vertices lying in $B$. 

ii) Let $G=(V_G, E_G)$ and $H=(V_H, E_H)$ be two directed hypergraphs with 
$V_G=\{1,2,3,4\}, E_G=\{\mathbf e_1=(\{1\}\to \{2,3\}), \mathbf e_2=(\{1\}\to \{2.4\})\}$, 
$V_H=\{a,b,c\}, E_H=\{\mathbf e_1^{\prime}=(\{a\}\to \{b,c\})\}$. The map $f_V\colon V_G\to V_H$,  
given by $f_V(1)=a, f_V(2)=b, f_V(3)=f_V(4)=c$,  induces a morphism $f$ of directed hypergraphs.  
However the map $f$ does not induce a morphism from $\mathfrak C^2(G)$ to $\mathfrak C^2(H)$. 
\end{example}

 For every morphism  $f=(f_V,f_E)\colon G\to H$ of directed hypergraphs, define a  map  
$
\mathfrak G(f)\colon V^d_G \to V^d_H \ \ \text{by}\  \mathfrak G(f)=f_V.
$
For any arrow  $(v\to w)\in E^d_G$,  we have  $(f_V(v)\to f_V(w))\in E^d_H$   and  the morphism  
$ \mathfrak G(f)$ of digraphs is well defined.   Thus we have  a functor 
$\mathfrak G$   from the category $\mathcal{DH}$ of directed hypergraphs to the category  $\mathcal D$ of digraphs.
\emph{Regular path homology of digraphs} was constructed in \cite{Axioms}, \cite{Mi3}. 
It is based on the natural functor  $\mathfrak D$  from the category $\mathcal D$ of 
digraphs to the category  $\mathcal P$ of path complexes.

\begin{theorem}\label{t3.3} For every directed hypergraph $G$ there is an
isomorphism 
$
H^{\mathbf c}_*(G)\cong H_*(\mathfrak{D}\circ\mathfrak{G}(G))
$
 of path homology groups.
\end{theorem}
\begin{proof} The path complexes $\mathfrak C(G)$ and $\mathfrak  D\circ \mathfrak G(G)$ coincide.   
\end{proof}

\begin{example}\label{e3.4}   \rm  The following example  illustrates the technique of computations of the connective path homology groups  $H_*^{\bold c(k)}(G)$.  
For 
$k\geq 3$ in the presented case, there 
is 
nothing to compute. Let $R=\mathbb R$ be the ring of coefficients. Consider a hypergraph $G=(V_G,E_G)$ for which 
   $
V_G=\{1,2, 3, 4\}, \ \ E_G=\{\bold e_1,\bold  e_2,\bold  e_3, \bold e_4, \bold e_5, \bold e_6\},
$
 $
\bold e_1=(\{1\}\to \{2\}), \bold e_2=(\{2\}\to \{3, 4\}), \bold e_3=(\{4\}\to \{1\}), 
$
$
\bold e_4=(\{1\}\to \{2,3\}), \bold e_5=(\{2\}\to \{3\}), \bold e_6=(\{2\}\to \{4\})$.
  
We compute homology of the path complex $\mathfrak C^c(G)=(V^c, P_G^c)$ as in \cite{Hyper}. We have 
$
\mathcal{R}_0^{reg}=\left<1, 2, 3, 4\right>=\Omega_0
$, 
$
\mathcal {R}_1^{reg}=\left<e_{12}, e_{13}, e_{23}, e_{24}, e_{41}\right>
$.
We get $\partial(e_{ij})\in \mathcal{R}_0^{reg}$ for all basic elements  $e_{ij}\in \mathcal{R}_1^{reg}$,  so 
$
\Omega_1=\mathcal{R}_1^{reg}
$.
Thus, $\Omega_1$ is generated by all directed edges of  the digraph $\mathfrak G(G)$ presented below
$$
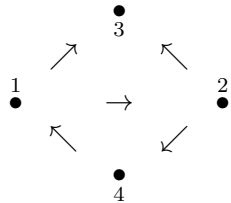
\begin{matrix}
&  & \underset{3}\bullet &&&&\\
 &\nearrow  &&\nwarrow&&&\\
\overset{1}\bullet &              &\to                          &  &\overset{2}\bullet& &\\
                           &\nwarrow  &&\swarrow&&&\\
&  & \underset{4}\bullet &&&&\\
\end{matrix}
$$
From the definition  of  $\mathfrak G(G)$, it follows that $\Omega_i=0$
for $i\geq 2$ and the homology of the chain complex $\Omega_*$ coincides with the regular path homology $\mathfrak G(G)$. 
Hence
 $
H_0^{\bold c(1)}(G)= H_1^{\bold c(1)}(G)=\mathbb R$ and $H_i^{\bold c(1)}(G)=0$ for $i\geq 2$. 

 For 
$H_i^{\bold c(2)}(G)$,   we have    
$\Omega_0=\left<1, 2, 3, 4\right>$ and by definition
$
\Omega_1=\left<e_{12},  e_{23}, e_{24}\right>
$.
Moreover, $\Omega_i=0$
for $i\geq 2$.  Thus,  homology groups $H_i^{\bold c(2)}(G)$ coincide with the homology groups 
of the digraph  which has the  set of vertices $V_{\mathfrak G(G)}$ and the set of arrows  obtained 
from $E_{\mathfrak G(G)}$ by deleting  arrows $(1\to 3)$ and 
$(4\to 1)$. Hence,   $
H_0^{\bold c(2)}(G)=\mathbb R$ and $H_i^{\bold c(2)}(G)=0$ for $i\geq 1$. 
For $k\geq 3$ we have 
 $
H_0^{\bold c(k)}=\mathbb R^4$ and $H_i^{\bold c(k)}(G)=0$ for $i\geq 1$.
\end{example}  

\begin{lemma}\label{l3.5} Let $G=(V,E)$ be a directed hypergraph and 
$I_1=(0\to 1)$. We have a natural isomorphism
$\mathfrak{C}(G\Box I_1)\cong[\mathfrak{C}(G)]^{\uparrow}$
of 
path complexes.
\end{lemma}
\begin{proof} By Definition \ref{d2.3} a directed hypergraph $G\Box I_1=(V_{G\Box I_1},E_{G\Box I_1}) $ 
has the set of vertices $V_{G\Box I_1}= V\times J=V\times \{0,1\}$ which we 
identify  with $V\cup V^{\prime}$,  where $V=\{0,\dots, n\}, \ V^{\prime}=\{0^{\prime}, \dots, n^{\prime}\}$ and  the set of edges  $E_{G\Box I_1}$ is the union  $E^0\cup E^1\cup  E^{01}$
of sets 
$E^i=\{A\times\{i\}\to B\times\{i\}\}$ with  $(A\to B)\in E_G$ for  $i=0,1$ 
and  $E^{01}=\{C\times\{0\}\to C\times\{1\}\}$ with  $C\in \mathbb{S}_{01}(G)$. Let 
$q=\left(i_0\dots i_n\right)$ be a path lying in $\mathfrak{C}(G\Box I_1)$. It follows from definition, that there are only three possibilities, namely

(1)  all the vertices $i_j\in V\times \{0\}$ and, hence, $q$ determines  the unique path in 
in $\mathfrak{C}(G)$, 

(2) all the vertices $i_j\in V\times \{1\}$ and, hence, $q$ determines the unique path in 
in $[\mathfrak{C}(G)]^{\prime}$, 

(3) there exists exactly one pair $(i_k,i_{k+1})$ of consequent vertices  in $q$ such that  $i_k\in C\times \{0\}, 
i_{k+1}\in C\times  \{1\}$ for $C\in \mathbb{S}_{01}(G)$. 

Thus, the union of paths  from (1)-(3) on the set  of vertices 
$V\times J$  defines the path complex $[\mathfrak{C}(G)]^{\uparrow}$ and vice versa. 
\end{proof}

\begin{theorem}\label{t3.6} For a directed hypergraph $G$,   the connective path homology groups $H^{\bold c}(G)$
are homotopy invariant.
\end{theorem}
\begin{proof} By Definition \ref{d2.4}, it is sufficient to prove 
 homotopy invariance for a one-step homotopy. Then the result follows from Lemma \ref{l3.5} and  \cite[Th. 3.4]{Hyper}.
\end{proof}

\subsection{Bold path homology}\label{S32}

Let  $p=\left(i_0\dots i_n\right)$ and $q=\left(j_0\dots j_m\right)$ be two paths of a  path complex 
$\Pi$   with $i_n=j_0$. The \emph{concatenation}  $p\vee q$  of these paths  is  a path  given by 
$
p\vee q=\left(i_0\dots i_nj_1\dots j_m\right)
$. 
The concatenation is well defined only if
 $i_n=j_0$.

For a  directed hypergraph   $G= (V, E)$, 
define a path complex  $\mathfrak B(G)=(V^b_G, P_G^b)$   where $V^b_G=V$ and  
a path $q=(i_0\dots i_n)\in P_{V}$ lies in $P_G^b$  iff  there is a sequence of 
 hyperedges $(A_0\to B_0), \dots, (A_r\to B_r)$  in $E$ such that   
 $B_i\cap A_{i+1}\ne \emptyset$ for $0\leq i\leq r-1$ and  the path $q$ has the 
  presentation 
\begin{equation}\label{3.2}
\left(p_0\vee v_0w_0 \vee p_1  \vee v_1w_1\vee p_2\vee \dots  \vee p_{r}\vee  v_r w_r\vee p_{r+1}\right)
\end{equation}
where $p_0\in P_{A_0}$,  $p_{r+1}\in P_{B_r}$,  $v_i\in A_i$, $w_i\in B_i$, $p_i\in P_{B_{i-1}}\cap  P_{A_{i}}$  
for  $1\leq i\leq r$  and all concatenations in (\ref{3.2}) are well defined.  Note, that in the case of empty 
sequence of edges $A_i\to B_i$ every  path  $q\in P_A$ and every path $q\in P_B$ for an edge $A\to B$ lies
in $P^b_G$. 

 \begin{proposition}\label{p3.7} Let  $f\colon G=(V_G, E_G)\to H=(V_H,E_H)$ be a  morphism of directed hypergraphs. 
Define a morphism  of path complexes
$$
\mathfrak B(f)=(f^b_V, f_p^b) \colon (V_G^b,P^b_G)\to (V_H^b, P^b_H)
$$ 
by  $f^b_V=f\colon V_G^b= V_G\to V_H=V_H^b$ and $f_p^b=f_p|{_{P^b_G}}$,  where 
$f_p$ is defined as in Proposition \ref{p3.1}.
 Thus,    we obtain  a functor $\mathfrak B$  from the category $\mathcal{DH}$ of directed hypergraphs to the category  $\mathcal P$ of path complexes.  \quad $ \blacksquare$
\end{proposition}

 Let us define  \emph{the bold path homology groups}  of  directed  hypergraph 
 $G$ by 
$
H_*^{\bold b}(G)\colon = H_*(\mathfrak B(G))
$.
By 
Proposition \ref{p3.7},  we  obtain a  functorial 
 path homology  theory on the category $\mathcal{DH}$ of directed hypergraphs. 

\begin{example}\label{e3.8} \rm   Let $G=(V, E)$ be a directed 
 hypergraph such that for  every edge $\bold e=(A\to B)\in E$ the sets $A$ and $B$ 
are one-vertex sets, 
$A=\{v\}, B=\{w\}, v,w\in V$. 
 We can consider the hypergraph $G$ as a digraph and
  $H^{\bold c}_*(G)\cong H^{\bold b}_*(G)$.
On the category of connected digraphs that can be considered as the
 subcategory of directed hypergraphs,   the bold path homology groups are 
 naturally isomorphic to the connective path homology groups and to the regular path
  homology groups  $H_*(G)$ defined in  \cite{Axioms}. 
\end{example} 

\begin{example}\label{e3.9} \rm Now we compute  the bold path homology groups 
	$H^{\bold b}_*(G)$ of  the directed hypergraph $G$  from  Example \ref{e3.4} in  dimensions 0,1,2 for $R=\mathbb R$. 
 First,  we describe the modules $\mathcal R_n^{reg}(\mathfrak B(G))$  for $0\leq n\leq 4$.  We have 
$$
\mathcal R_0^{reg}=\langle e_1,e_2,e_3,e_4\rangle, \ \
\mathcal R_1^{reg}=\langle e_{12}, e_{13},
e_{23},e_{24},
e_{32}, e_{34}, 
e_{43},e_{41}\rangle,
 $$
$$
\mathcal R_2^{reg}=\langle e_{123},e_{124},
 e_{132},
 e_{232}, e_{234}, 
e_{241}, e_{243},
e_{323},
 e_{343},
 e_{434},
e_{412}, e_{413}\rangle,
 $$
$$
\begin{matrix}
\mathcal R_3^{reg}=\langle e_{1232},e_{1234}, 
e_{1241},   e_{1243},
e_{1323},
e_{2323},   
e_{2343},\\
e_{2412}, e_{2413},
e_{2434},
e_{3232},
e_{3434}, 
e_{4343},
e_{4123}, e_{4124},
e_{4132}
\rangle,
\end{matrix}
 $$
$$
\begin{matrix}
\mathcal R_4^{reg}=
\langle e_{12323},
e_{12343}, 
e_{12412},  e_{12413}, 
 e_{12434},
e_{13232},
e_{23232},   
e_{23434},\\
e_{24123},
 e_{24132},
e_{24343},
e_{32323},
e_{34343}, 
e_{43434},
e_{41232},e_{41234},
 e_{41243},
e_{41323}
\rangle.
\end{matrix}
 $$
$$
 \Omega_n=\mathcal{R}_n^{reg} \ \ \text{for} \ \ n=0,1.
$$
Thus $\Omega_0$ is generated by all the vertices and $\Omega_1$ is generated 
by all directed edges of  the digraph $H$ on Fig. 1.  

\begin{figure}[th]\label{fig1}
\centering
\begin{tikzpicture}
\node (1) at (4,3) {$1$};
\node (2) at (8,3) {$2$};
\node (3) at (6,5) {$3$};
\node (4) at (6,2) {$4$};
\draw (1) edge[ thick, ->] (3);
\draw (4) edge[ thick, ->] (1);
\draw (2) edge[thick, ->] (4);
\draw (3) edge[bend right=12, thick, ->]  node [left]{} (2);
\draw (2) edge[bend right=15, thick, ->] node [right]{} (3);
\draw (1) edge[bend right=90,  thick, ->] node [right]{} (2);
\draw (3) edge[bend right=15,  thick, ->] node [right]{} (4);
\draw (4) edge[bend right=15,  thick, ->] node [right]{} (3);
\end{tikzpicture}
\caption{The digraph $H$.}
\end{figure}
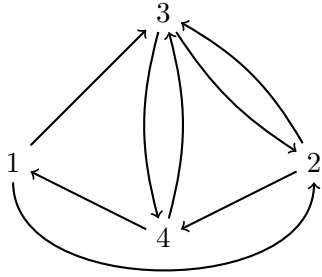
As it follows from the path homology theory of digraphs, the rank 
of the image 
$\partial\colon   \Omega_1\to \Omega_0$ is equal to 3, 
the rank 
of the kernel $\partial$ is equal to 5,  
and hence  $H^{\bold b}_0(G)=\mathbb R$.
  
By the direct computation $\Omega_2$ is the vector space with the following basis:
$
\{
e_{123},
e_{132},
e_{232}, 
e_{234},
 e_{243},
e_{323}, 
e_{343},
e_{434},
e_{413}
\}
 $.
In this basis the matrix of homomorphism $\partial \colon \Omega_2\to \Omega_1$ has the form:
$$
\left(\begin{matrix}
            &e_{12}&e_{13}&e_{23}&e_{24}&e_{32}&e_{34}&e_{43}&e_{41}\\
e_{123}&1         &-1       &1        &0         &0        &0        &0        &0        \\
e_{132}&-1         &1       &0        &0         &1        &0        &0        &0        \\
e_{232}&0        &0       &1        &0         &1       &0        &0        &0        \\
e_{234}&0        &0      &1        &-1        &0        &1       &0        &0        \\
e_{243}&0        & 0     &-1        &1       &0        &0        &1        &0        \\
e_{323}& 0       &0      &1        &0         &1      &0        &0        &0        \\
e_{343}&0        &0       &0        &0         &0        &1       &1        &0        \\
e_{434}&0        &0       &0        &0         &0        &1       &1        &0        \\
e_{413}&0        &1      &0        &0         &0        &0       &-1      &1        \\
\end{matrix}\right).
$$
Its rank is equal to 5. 
Hence the rank 
of the image of
$\partial$ is equal to 5, 
the rank 
of the kernel $\partial$ is equal to 4,  
and hence  $H^{\bold b}_1(G)=0$.  

We have 
$
\begin{matrix}
\Omega_3=\langle e_{1232},
e_{1323},
e_{2323},   
e_{2343},
e_{2434},
e_{3232},
e_{3434}, 
e_{4343}
\rangle.\\
\end{matrix}
 $
Similar to the previous calculation,  the rank 
of the image 
$\partial\colon   \Omega_3\to \Omega_2$ is equal to 4, 
the rank 
of the kernel $\partial$ is  equal to 4,  
and hence  $H^{\bold b}_2(G)=0$.  

We have 
$\Omega_4=\langle e_{12323}, e_{13232}, e_{23232}, 
e_{23434}, 
e_{24343}, e_{32323}, e_{34343}, e_{43434}\rangle 
 $ 
 and, similar to the previous calculation, the  rank 
of the image 
$\partial\colon   \Omega_4\to \Omega_3$ is equal to 4, 
the rank 
of the kernel $\partial$ is equal to 4.
Hence  $H^{\bold b}_3(G)=0$.  
\end{example} 

\begin{lemma}\label{l3.10}  Let $G=(V,E)$ be a directed hypergraph and 
$I_1=(0\to 1)$  the digraph. There is an inclusion 
$\lambda\colon [\mathfrak{B}(G)]^{\uparrow}\to \mathfrak{B}(G\Box I_1)$
of path complexes. The restrictions of 
$\lambda$ to the  images of the morphisms $i_{\bullet}$ and $j_{\bullet}$,   defined in Section \ref{S2}, 
are the natural identifications. 
\end{lemma}
\begin{proof} By definition in Section \ref{S2}, 
  we have
$
[\mathfrak{B}(G)]^{\uparrow}=(V\times J, [P^b_G]^{\uparrow}))$, 
where $[P^b_G]^{\uparrow}=P^b_G\cup [P^b_G]^{\prime}\cup [P^b_G]^{\#} $.  
We have  $V_{G\Box I_1}= V\times J=V\times \{0,1\}=V\cup V^{\prime}$ with
$V=\{0,\dots, n\}, \ V^{\prime}=\{0^{\prime}, \dots, n^{\prime}\}$ and   
$E_{G\Box I_1}$ is  the union  of sets $E^0\cup E^1\cup  E^{01}$,  where
$
E^i=\{A\times\{i\}\to B\times\{i\} \, | \, (A\to B)\in E_G\}
$
 for  $i=0,1$ and 
$
E^{01}=\{C\times\{0\}\to C\times\{1\} \, | \, C\in {\mathbf{P}}_{01}(G))\}
$.  
Now it follows
that 
$
\mathfrak{B}(G)=\mathfrak{B}(G\Box\{0\}), \mathfrak{B}(G)^{\prime}=\mathfrak{B}(G\Box\{1\})
$, 
where $\mathfrak{B}(G\Box\{0\}),\mathfrak{B}(G\Box\{1\})\subset \mathfrak{B}(G\Box I_1)$.
Let $q=(i_0\dots i_n)$ be $n$-path in $\mathfrak{B}(G)=\mathfrak{B}(G\Box\{0\})$. Consider  
its presentation in the form (\ref{3.2}) and let   $A_i, B_i$ be the corresponding sets 
of vertices. 
 For   $0\leq k \leq n$,  consider a path 
 $q_k^{\#}=\left(i_0\dots i_ki_k^{\prime}i^{\prime}_{k+1}\dots i_n^{\prime}\right)\in [P^b_G]^{\#}$.  
  We will prove now that this path  in 
$P^b_{G\Box I_1}$. There are  following possibilities for the path $q$.

(1) Vertices
$i_k, i_{k+1}\in p_s$  for $1\leq s\leq r+1$ in presentation (\ref{3.2}). 
Then we  
write  path $q_k^{\#}$ in the form 
\begin{equation}\label{3.3}
q_k^{\#}=\left(p_0^{\#}\vee v_0^{\#}w_0^{\#} \vee p_1^{\#} \vee  \dots  \vee 
p_{r+1}^{\#}\vee  v_{r+1}^{\#} w_{r+1}^{\#}\vee
 p_{r+2}^{\#}\right)
\end{equation}
  putting 
$$
A_i^{\#}=\begin{cases}  A_i\times\{0\} & \text{for} \ i\leq s-1\\
 B_{s-1}\times\{0\} & \text{for} \ i=s\\
A_{s-1}\times\{1\} & \text{for} \ i\geq s+1, \\
\end{cases}\ \ \
B_i^{\#}=\begin{cases}  B_i\times\{0\} & \text{for} \ i\leq s-1\\
B_{s-1}\times\{1\} & \text{for} \ i\geq s. \\
\end{cases}
$$
We have the following arrows  in  $E_{G\Box I_1}$:
$$
(A_i^{\#}\to B_i^{\#})= 
(A_i\times \{0\}\to  B_i\times\{0\}) \ \ \text{for} \  0\leq i\leq s-1, 
$$
$$
(A_s^{\#}\to B_s^{\#})= 
(B_{s-1}\times \{0\}\to  B_{s-1}\times \{1\}), 
$$
$$
(A_i^{\#}\to B_i^{\#})= 
(A_{i-1}\times \{1\}\to  B_{i-1}\times\{1\}) \ \ \text{for} \  s+1\leq i\leq r+2. 
$$
 Using identifications 
$
\mathfrak{B}(G)=\mathfrak{B}(G\Box\{0\}), \mathfrak{B}(G)^{\prime}=\mathfrak{B}(G\Box\{1\})
$,  
we obtain
\begin{equation}\label{3.4}
p_i^{\#}= \begin{cases}p_i& \text{for}  \  i\leq s-1\\
p_i^{\prime}& \text{for}  \  s+2\leq i\leq r+2\\
(w_{s-1}\dots i_s) &\text{for}  \  i=s\\
(i_{s}^{\prime}i_{s+1}^{\prime} \dots v_s^{\prime})&\text{for}  \  i=s+1\\
\end{cases}
\end{equation}
where 
$
(w_{s-1}\dots i_s)\in P_{B_{s-1}^{\#}\cap A_s^{\#}}=
P_{B_{s-1}\times\{0\}}
$, 
$
\left(i_{s}^{\prime}i_{s+1}^{\prime} \dots v_s^{\prime}\right)
\in P_{B_{s}^{\#}\cap A_{s+1}^{\#}}=
P_{(B_{s-1}\times\{1\})\cap (A_{s}\times\{1\})}
$.
Paths $p_i^{\#}$ in (\ref{3.4}) 
 define  vertices 
$v_i^{\#}, w_{i}^{\#}$ in (\ref{3.3}). 
Hence,   (\ref{3.3}) gives a presentation of $q_k^{\#}$ in the form (\ref{3.2}) 
for the hypergraph $G\Box I_1$  and  $q_k^{\#}\in P_{G\Box I_1}^b$ in the considered case. 

(2)  
Vertices  $i_k,i_{k+1}\in p_0$ in presentation (\ref{3.2}).  Then we  
write  
path $p_k^{\#}$ in the form  
(\ref{3.3})  putting 
$$
A_i^{\#}=\begin{cases}  A_0\times\{0\} & \text{for} \ i=0\\
 A_{i-1}\times\{1\} & \text{for} \ 1\leq  i\leq r+2, \\
\end{cases} 
$$
$$
B_i^{\#}=\begin{cases}  A_0\times\{1\} & \text{for} \ i=1\\
B_{i-1}\times\{1\} & \text{for} \ 2\leq  i\leq r+2 \\
\end{cases}
$$
and  
$$
p_i^{\#}= \begin{cases}(i_0\dots i_k)& \text{for}  \  i=0\\
\left(i_k^{\prime}\dots v_0^{\prime}\right)& \text{for}  \  i=1\\
p_{i-1}^{\prime} &\text{for}  \ 2\leq i \leq r+2\\
\end{cases}
$$
where 
$(i_0\dots i_k)\in P_{A_0^{\#}}$,  $\left(i_k^{\prime}\dots v_0^{\prime}\right)\in 
P_{B_0^{\#}\cap  A_1^{\#}}=P_{A_{0}\times\{1\}}$
Hence,   (\ref{3.3}) gives a presentation of $q_k^{\#}$ in the form (\ref{3.2})
in the hypergraph $G\Box I_1$  and  $q_k^{\#}\in P_{G\Box I_1}^b$ in the considered case.

(3) Let $i_k=v_s, i_{k+1}=w_s$  for $0\leq s\leq r$ in the presentation (\ref{3.2}). 
 Then we   
 write   
 path $q_k^{\#}$ in the form  
(\ref{3.3})  putting 
$$
A_i^{\#}=\begin{cases}  A_i\times\{0\} & \text{for} \ i\leq s\\
 A_{s+1}\times\{1\} & \text{for} \ i=s+1\\
A_{s-1}\times\{1\} & \text{for} \ i\geq s+2, \\
\end{cases}\ 
B_i^{\#}=\begin{cases}  B_i\times\{0\} & \text{for} \ i\leq s-1\\
A_s\times\{1\} & \text{for} \ i=s\\
B_{s-1}\times\{1\} & \text{for} \ i\geq s+1. \\
\end{cases}
$$
We have the following arrows  in $E_{G\Box I_1}$:
$$
(A_i^{\#}\to B_i^{\#})= 
(A_i\times \{0\}\to  B_i\times\{0\}) \ \ \text{for} \  0\leq i\leq s-1, 
$$
$$
(A_s^{\#}\to B_s^{\#})= 
(A_{s}\times \{0\}\to  A_{s}\times \{1\}), 
$$
$$
(A_i^{\#}\to B_i^{\#} )= 
(A_{i-1}\times \{1\}\to  B_{i-1}\times\{1\}) \ \ \text{for} \  s+1\leq i\leq r+2. 
$$
Similarly to  case (1),  we have  
$$
p_i^{\#}= \begin{cases}p_i& \text{for}  \  i\leq s\\
(v_s^{\prime})& \text{for}  \  i=s+1\\
p_{i-1}^{\prime} &\text{for}  \ s+2\leq i \leq r+2,\\
\end{cases}
$$
where 
$
w_{s}^{\#}=v_s, v_{s}^{\#}=v_s^{\prime}, 
w_{s+1}^{\#}=v_s^{\prime}$.
Hence,   (\ref{3.3}) gives a presentation of $q_k^{\#}$ in the form (\ref{3.2})
in  the hypergraph $G\Box I_1$  and  $q_k^{\#}\in P^b_{G\Box I_1}$ in the considered case.  
\end{proof}

\begin{theorem}\label{t3.11}  Let $G$ be a directed hypergraph. The bold  path homology groups $H^{\bold b}_*(G)$
are homotopy invariant.
\end{theorem}
\begin{proof} By Definition \ref{d2.4}, it is sufficient to prove 
 homotopy invariance for  the one-step homotopy.  Let $f_0, \, f_1\colon G\to H$ be 
  one-step homotopic morphisms of directed hypergraphs with  homotopy 
$F\colon G\Box I_1\to H$,  where $I_1=(0\to 1)$. 
Since $\mathfrak B$ is a functor,  we obtain
morphisms of path complexes
$
\mathfrak B(f_0), \, \mathfrak B(f_1)\colon \mathfrak B(G)\to \mathfrak B(H)$
and $\mathfrak B(F)\colon \mathfrak B(G\Box I_1)\to \mathfrak B(H)$. Consider the composition 
$
[\mathfrak{B}(G)]^{\uparrow}\overset{\lambda}{\longrightarrow} \mathfrak{B}(G\Box I_1)\overset{\mathfrak B(F)}{\longrightarrow}
\mathfrak B(H)
$
which gives a homotopy between morphisms 
$\mathfrak B(f_0)$ and $\mathfrak B(f_1)$ of
path complexes by using identifications of the top and the bottom 
of $[\mathfrak{B}(G)]^{\uparrow}$
described in Lemma \ref{l3.10}.
Now the result follows from 
\cite[Th. 3.4]{Hyper}.
\end{proof}

\subsection{Non-directed path homology}\label{S33}

In this subsection, we describe several path  homology theories on the category of directed  hypergraphs  
$\mathcal{DH}$ that are based on    functorial relations between  hypergraphs and directed  hypergraphs. 

 For  a hypergraph $G=(V_G,E_G)$,  we have  a natural map   
 $\phi_G\colon  E_G \to  {\mathbf{P}}(V_G)\setminus \emptyset$.
 A  \emph{morphism} of hypergraphs $f\colon G\to H=(V_H,E_H)$   
 is given by the pair of maps  $f_V\colon V_G\to V_H$ 
and $f_E\colon E_G\to E_H$ 
 such that  ${\mathbf{P}}(f_V)\circ \phi_G=\phi_H\circ f_E$,  
where 
${\mathbf{P}}(f_V)\colon{\mathbf{P}}(V_G)\setminus \emptyset \to {\mathbf{P}}(V_H)\setminus \emptyset$ 
is the map induced by $f_V$. So  we may turn to 
the category   of hypergraphs  $\mathcal H$ in \cite{Hyper}.

First,  define  a functor from  category $\mathcal{DH}$
 to category $\mathcal H$. For a finite set $X$, define  a map $\sigma_X\colon \mathbb P(X)\to {\mathbf{P}}(X)$ by setting $\sigma_X(A,B)=A\cup B$.   
 Let $G=(V, E)$ be a directed hypergraph. 
Define a hypergraph $\mathfrak E(G)=(V^e, E^e)$ where  $V^e=V$ and 
\begin{equation}\label{3.5}
E^{e}=\{C\in {\mathbf{P}}(V)\setminus \emptyset \, | \, C=A\cup B, (A\to B)\in E\}.
\end{equation}
Recall that in Section \ref{S2},  for a directed hypergraph $G=(V,E)$ we defined a map $\varphi_G\colon E\to \mathbb P(V)$ by $\varphi_G(A\to B)=
(A,B)$.

 \begin{proposition}\label{p3.12} Let  $f=(f_V, f_E)\colon G=(V_G, E_G)\to H=(V_H,E_H)$ be a  morphism of directed hypergraphs. 
Define a map $f_E^e\colon  E^e_G\to {\mathbf{P}}(V_H)$ 
putting 
$
f_E^e(C)=[{\mathbf{P}}(f_V)](C)
$
for every  $C=A\cup B\in E^e_G$. 
Then the map $f_E^e$ is a well defined map 
$  E^e_G\to  E^e_H$  and the pair 
$
(f_{V}^e, f_{E}^e)$ with $f_{V}^e=f_V$
defines  a morphism 
$
\mathfrak E(f) \colon  (V_G^e,E^e_G)\to (V_H^e, E^e_H)
$
 of hypergraphs.  Thus,    we obtain  a functor $\mathfrak E$  from  the category 
 $\mathcal{DH}$ of directed hypergraphs to the category  $\mathcal H$ of hypergraphs. 
\end{proposition}
\begin{proof} The map $f_E^e$ is 
	well defined. Now we prove that its image lies in $E^e_H$. 
Let  $C\in E^e_G$,  $C=\sigma_{V_G}\circ \varphi_G(A\to B)=A\cup B$ and 
$f_{E_G}(A\to B)=(A^{\prime}\to B^{\prime})\in E_H$. Then,  by Definition \ref{d2.2}, 
$[\mathbb{P}(f_{V_G})](A,B)= (A^{\prime},  B^{\prime})\in\mathbb{P}(V_G))$ and, hence, 
$A^{\prime}=[{\mathbf{P}}(f_V)](A), B^{\prime}=[{\mathbf{P}}(f_V)](B)$. 
We have 
$$
[{\mathbf{P}}(f_V)](C)= [{\mathbf{P}}(f_V)](A\cup B)=\{ [{\mathbf{P}}(f_V)](A)\}\cup 
\{ [{\mathbf{P}}(f_V)](B)\}=A^{\prime}\cup B^{\prime}. 
$$
However, 
$
A^{\prime}\cup B^{\prime}=\sigma_{V_H}\circ \varphi_H(A^{\prime}\to B^{\prime})\in E^e_H
$
and the claim  that morphism $f^e_E$ is well defined is proved. 
The 
functoriality is evident.
\end{proof} 

For a hypergraph   $G= (V, E)$, 
define a path complex  $\mathfrak H^{q}(G)=(V^{q}, P_G^{ q})$ of \emph{density} $q\geq 1$
 where $V^{ q}=V$ and   a path $(i_0\dots i_n)\in P_{V}$ lies $\in P_G^{ q}$  
 iff  every $q$  consequent vertices of this path
lie in a  hyperedge $\mathbf e$, see \cite{Hyper}. 
Thus,  we obtain a collection of functors $\mathfrak H^{ q}$  from  the category $\mathcal H$ to  the category  $\mathcal P$.  
Composition  $\mathfrak H^{ q}\circ \mathfrak E$ gives collection of functors   
from category $\mathcal{DH}$  to category  $\mathcal P$.
 For a  directed hypergraph $G$ define 
\begin{equation*}
H_*^{\bold e(q)}(G)\colon = H_*(\mathfrak H^{q}\circ \mathfrak E(G))
\ \text{for} \ \ q=1,2,\dots \ \ .
\end{equation*}
We  call these groups by  the \emph{non-directed path homology groups of density 
	$q$}  of a directed  hypergraph $G$. We 
denote $H_*^{\bold e}(G)\colon =H_*^{\bold e(1)}(G)$.

\begin{proposition}\label{p3.13} Let $G=(V,E)$ be a directed hypergraph and 
	$\Pi_V$ be a path complex of all paths on the set $V$. Then  
	$H_*^{\bold e}(G)=H_*(\Pi_V)$.
\end{proposition}
\begin{proof} By Definition \ref{d2.1}, $V=\cup_{\bold e_i\in E} (A_i\cup B_i)$ 
	and every vertex $v\in V^e=V$ lies in an edge $e\in E^e$. So path complexes 
$\Pi_V$ and $\mathfrak H^{1}\circ \mathfrak E(G)$ coincide.
\end{proof}
\begin{example}\label{e3.14} \rm  Now we compute   path homology groups 
$H^{\bold e(q)}_*(G)$ of density  $q=1,2,3$ with coefficients in $\mathbb R$ of the  
 directed hypergraph $G$ with 
   $
V_G=\{1,2, 3, 4,5,6\}, \ \ E_G=\{\bold e_1,\bold  e_2,\bold  e_3, \bold e_4, \bold e_5\},
$
 where 
 $
\bold e_1=(\{1\}\to \{2\}), \bold e_2=(\{1\}\to \{3\}), \bold e_3=(\{2\}\to \{4,6\}), 
\bold e_4=(\{3\}\to \{5\}), \bold e_5=(\{4\}\to \{5,6\})
$.
Then  the hypergraph $\mathfrak E(G)$ has the set of vertices
$V_G^e=\{1,2, 3, 4,5,6\}$ and the set of hyperedges 
$$
E_G^e=\left\{ 
\bold e_1^{\prime}=\{1,2\}, 
\bold e_2^{\prime}=\{1, 3\}, 
\bold e_3^{\prime}=\{2,4,6\}, 
\bold e_4^{\prime}=\{3, 5\},
 \bold e_5^{\prime}=\{4,5,6\}\right\}.
$$ 
In the case of $q=1$, the homology groups $H_*^{\mathbf e}(G)$ 
coincide with the path homology group of the complete  digraph 
$D=(V_D, E_D)$ which has six vertices and for every two vertices 
$v,w\in V_D$  there are two arrows $(v\to w), (w\to v)\in E_D$. 
This digraph is contractible, and hence, see  \cite[S3.3]{MiHomotopy}, 
$
H_0^{\bold e}(G)= \mathbb R$ and groups $H_i^{\bold e}(G)$ are trivial for $i\geq 1$.
\begin{figure}[th]\label{fig2}
\centering
\begin{tikzpicture}
\node (1) at (5,3) {$1$};
\node (3) at (7,3) {$3$};
\node (5) at (9,3) {$5$};
\draw (1) edge[ thick, ->] (3)
(3) edge[ thick, ->] (1)
 (3) edge[ thick, ->] (5)
(5) edge[ thick, ->] (3);
\node (2) at (5,5) {$2$};
\node (4) at (7,5) {$4$};
\node (6) at (9,5) {$6$};
\draw  (4) edge[ thick, ->] (6);
\draw  (6) edge[ thick, ->] (4)
(2) edge[ thick, ->] (4)
(4) edge[ thick, ->] (2)
(5) edge[ thick, ->] (4)
(4) edge[ thick, ->] (5)
 (1) edge[ thick, ->] (2)
 (2) edge[ thick, ->] (1)
 (5) edge[ thick, ->] (6) 
 (6) edge[ thick, ->] (5) 
(2) edge[bend left,  thick, ->] node [right]{} (6) 
(6) edge[bend right,  thick, ->] node [right]{} (2);
\end{tikzpicture}
\caption{The digraph $D_2$ for $q=2$.}
\end{figure}
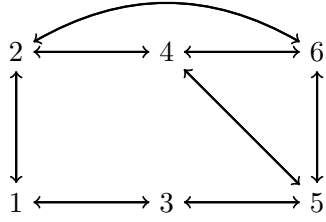

\noindent
If $q=2$,  homology groups $H_*^{\mathbf e(2)}(G)$ 
coincide with path homology group of the  digraph $D_2$ on Fig. 2,  where  two-sided arrow   
$a\longleftrightarrow b$ means that there are arrows 
$a\to b$ and $b\to a$.  The digraph $D$ is homotopy equivalent to the induced 
sub-digraph $D_2^{\prime}\subset D_2$ with the set of vertices $\{1,2,3,4,5\}$. 
 We  compute directly the path homology of 
$D_2^{\prime}$ and we obtain  
$ 
H_0^{\bold e(2)}(G)=H_1^{\bold e(2)}(G)=\mathbb R$ and  trivial groups 
$H_i^{\bold e}(G)$  for $i\geq 2$.

Now we consider the case of $\mathbf e(3)$.
We have  $\Omega^{\mathbf e(3)}_n=\Omega^{\mathbf e(2)}_n$  for $n=0,1$ and 
this equality is also true  for all $n\geq 0$.  We have 
$$
 {\mathcal{R}_2^{\mathbf e(3)}}^{reg}= A \oplus A_{246}\oplus A_{456},
 $$
where $A=\langle e_{121}, 
e_{212}, e_{131}, e_{313}, e_{353},e_{535}\rangle$ and
$A_{abc}$ is  the module generated by all regular paths with three vertices  
in the full digraph with vertices $a,b,c$.
Hence $\Omega_2^{\mathbf e(3)}={\mathcal{R}_2^{\mathbf e(3)}}^{reg}$.  
Considering  the digraph $D_2$, 
we obtain that  $\Omega_2^{\mathbf e(3)}= 
\Omega_2^{\mathbf e(2)}$. The cases with 
$n\geq 4$ are similar and  $\Omega_n^{\mathbf e(3)}= 
\Omega_n^{\mathbf e(2)}$ for $n\geq 4$.  Hence,  
$ 
H_n^{\bold e(2)}(G)=H_n^{\bold e(3)}(G)$ for $n\geq 0$. 
\end{example}
\begin{proposition}\label{p3.15} Let $G=(V,E)$ be a directed hypergraph, 
$I_1=(0\to 1)$,
and $ I=(V_I, E_I)$ be the hypergraph  
with the set of vertices $V_I=\{0,1\}$ and the set of edges  
$ E_I=\{ \bold e_0^{\prime}=\{0\}, \bold e_1^{\prime}=\{1\}, \bold e_2^{\prime}=\{0,1\}\}$. 
There is a natural inclusion  of path complexes
\begin{equation}\label{3.6}
\mathfrak H^{ q}[\mathfrak{E}(G\Box I_1)] \subset \mathfrak H^{ q}[\mathfrak{E}(G)\times I]
\end{equation}
for $q\geq 2$. Moreover,  in general case  complexes in (\ref{3.6}) are not equal. 
\end{proposition}
\begin{proof} Recall that the product $"\times"$ of hypergraphs is defined in \cite{mor}, \cite{Hyper}. 
The directed hypergraph  
 $G\Box I_1=(V_{G\Box I_1}, E_{G\Box I_1}) $ has the set of vertices  
 $V_{G\Box I_1}=V\times \{0,1\}$ and the set of edges that can be presented as the  union  
 $E_0\cup E_1\cup E_{01}$ of three pairwise disjoint sets 
\begin{equation*}
\begin{matrix}
E_0=\{(C\times \{0\}\to  D\times \{0\})| (C\to D)\in E\},
\\
E_1=\{(C\times \{1\}\to  D\times \{1\})| (C\to D)\in E\},
\\
E_{01}=\{A\times\{0\}\to A\times\{1\}| A\subset {\mathbf{P}}_{01}(G) \}.\ \ \ \  \\
\end{matrix}
\end{equation*}
Hence,  the hypergraph 
$
\mathfrak E(G\Box I_1)=\left( V_{\mathfrak E(G\Box I_1)}, E_{\mathfrak E(G\Box I_1)}  \right)
$
has the set of vertices  $V_{\mathfrak E(G\Box I_1)}=  V_{G\Box I_1}=V\times \{0,1\}$ and the set of edges  
that can be presented as a union of three pairwise disjoint sets $E_0^{\prime}\cup E_1^{\prime}\cup E_{01}^{\prime}$ where
\begin{equation}\label{3.7}
\begin{matrix}
E_0^{\prime}=\{(C\times \{0\})\cup (D\times \{0\})| (C\to D)\in E\},
\\
E_1^{\prime}=\{(C\times \{1\})\cup  (D\times \{1\})| (C\to D)\in E\},
\\
E_{01}^{\prime}=\{(A\times\{0\})\cup( A\times\{1\})| A\subset {\mathbf{P}}_{01}(G) \}.
\ \     \\
\end{matrix}
\end{equation}
By definition of a hypergraph  $\mathfrak{E}(G)=(V^e, E^e)$,   we obtain that the hypergraph 
 $\mathfrak{E}(G)\times I=(V_{\mathfrak{E}(G)\times I}, E_{\mathfrak{E}(G)\times I}) $ has the set of vertices 
$V_{\mathfrak{E}(G)\times I}=V\times \{0,1\}$ and the set of edges that can be presented as  
a union of three pairwise disjoint sets 
$E_0^{\prime\prime}\cup E_1^{\prime\prime}\cup E_{01}^{\prime\prime}$ where
\begin{equation}\label{3.8}
\begin{matrix}
E_0^{\prime\prime}=\{(C\cup D) \times \{0\}, C\cup D, \{0\})\, | \,(C\to D)\in E\},\\
E_1^{\prime\prime}=\{(C\cup D) \times \{1\}, C\cup D, \{0\})\, |\,  (C\to D)\in E\},\\
E_{01}^{\prime\prime}=\{(A, C\cup D, \{0,1\})\, | \, (C\to D)\in E, A\subset (C\cup D) \times \{0,1\} \}.\\
\end{matrix}
\end{equation}
Let $p_1\colon V\times \{0,1\}\to V,  \,
p_2\colon V\times \{0,1\}\to \{0,1\}$ be natural projections. 
Then  $p_1(A)=C\cup D$ and  $p_2(A)=\{0,1\}$ by definition of the product of hypergraphs.  
Thus,  path complexes $\mathfrak H^{q}\left[\mathfrak E(\Box I_1)\right]$  and  
$\mathfrak H^{ q}\left[\mathfrak E(G)\times  I\right]$ have the same vertex set   and,  by  (\ref{3.7})  and (\ref{3.8}), 
$
E_0^{\prime}=E_0^{\prime\prime},\ 
 E_1^{\prime}= E_1^{\prime\prime}, \ E_{01}^{\prime}\subset E_{01}^{\prime\prime}
$
and (\ref{3.6}) follows. 

Now we prove that in general case  of (\ref{3.6}) there is no equality. 
Let $(C\to  D)\in E$ be a directed edge 
and $v\in C, w\in D$ be such vertices that the pair $(v,w)$ does not lie in a set $A\in {\mathbf{P}}_{01}(G)$. 
Then  the two vertex  path $(v\times\{0\}), 
(w\times\{1\})\in\mathfrak H^{q}\left[\mathfrak E(G)\times  I\right]$  for $ q=2$ lies in
$E_{01}^{\prime\prime}$ and does not lie in  $E_0^{\prime}\cup
 E_1^{\prime}\cup E_{01}^{\prime}$.
\end{proof}
\begin{lemma}\label{l3.16}  Let $G=(V,E)$ be a directed hypergraph,
$I_1=(0\to 1)$.
There is the  inclusion 
$
\mu\colon [\mathfrak H^{ 2}\circ\mathfrak{E}(G)]^{\uparrow}\to 
\mathfrak{H}^{2}\circ \mathfrak E(G\Box I_1)
$
of path complexes. 
\end{lemma}
\begin{proof} By definition of the hypergraph  $\mathfrak{E}(G)=(V^e, E^e)$ and  the functor $\mathfrak H^2$,   
	we obtain that the set  $P_{\mathfrak{E}(G)}^2$ of $\mathfrak H^{ 2}\circ\mathfrak{E}(G)$ consists 
of  paths $p=(i_0\dots i_n)$ on the set $V$ such that every two 
consequent vertices $i_s,i_{s+1}\in p$ lie in a set 
$\{C\cup D\, | \, (C\to D)\in E\}$. By definition,  the set of paths of  the path
complex  $[\mathfrak H^{ 2}\circ\mathfrak{E}(G)]^{\uparrow}$
is a union of paths 
\begin{equation}\label{3.9}
P_{\mathfrak{E}(G)}^2\cup [P_{\mathfrak{E}(G)}^2]^{\prime}\cup {[P_{\mathfrak{E}(G)}^2]}^{\#}
\end{equation}
on the set $V\times\{0,1\}=V\cup V^{\prime}$. 
A path $p=(i_0\dots i_n)$ on the set $V\times \{0,1\}$ lies in  
$P_{\mathfrak E(G\Box I_1)}^2$ if any two consequent vertices lie in the exactly one of the sets 
$
E_0^{\prime}, E_1^{\prime}, E_{01}^{\prime}
$
 defined in (\ref{3.7}).  From definition of the functor $\mathfrak{E}$, we conclude that  in (\ref{3.9}) 
 any pair of consequent vertices of a path  from $P_{\mathfrak{E}(G)}^2$ lies in an  edge from $E_0^{\prime}$, 
any pair of consequent vertices of a path from $[P_{\mathfrak{E}(G)}^2]^{\prime}$ lies in an  edge from 
$E_1^{\prime}$, and any pair of consequent vertices of a path from 
${[P_{\mathfrak{E}(G)}^2]}^{\#}$  lies in an  edge from $E_0^{\prime}\cup  E_1^{\prime}\cup E_{01}^{\prime}$.  
\end{proof} 
\begin{theorem}\label{t3.17}  For a directed hypergraph $G$,  
the non-directed  path homology groups $H^{\bold e(2)}_*(G) $
  of density two  are homotopy invariant.
\end{theorem}
\begin{proof} It is sufficient to prove 
 homotopy invariance for the  one-step homotopy. 
Let $f_0, \, f_1\colon G\to H$ be  one-step homotopic morphisms of directed hypergraphs with a homotopy 
$F\colon G\Box I_1\to H$.  Since $\mathfrak H^{2}\circ \mathfrak E$ is a functor,  we obtain
morphisms of path complexes
$
\mathfrak H^{2}\circ \mathfrak E(f_0), \, 
\mathfrak H^{2}\circ \mathfrak E(f_1)\colon
 \mathfrak H^{2}\circ \mathfrak E(G)\to \mathfrak H^{2}\circ \mathfrak E(H), 
$
and 
$
\mathfrak H^{2}\circ \mathfrak E(F)\colon
\mathfrak H^{2}\circ \mathfrak E(G\Box I_1)\to 
\mathfrak H^{2}\circ \mathfrak E(H)
$.
 Using Lemma \ref{l3.16},  we can consider the composition 
$$
[\mathfrak H^{2}\circ \mathfrak E(G)]^{\uparrow}\overset{\mu}
{\longrightarrow} \mathfrak H^{2}\circ \mathfrak E(G\Box I_1)\overset{\mathfrak H^{2}\circ \mathfrak E(F)}{\longrightarrow}
\mathfrak H^{2}\circ \mathfrak E(H)
$$
which gives a homotopy between morphisms 
$ \mathfrak H^{2}\circ \mathfrak E(f_0)$ and 
$
\mathfrak H^{2}\circ \mathfrak E(f_1)$.
Now the result follows from 
\cite[Th. 3.4]{Hyper}.
\end{proof}

\subsection{Natural path homology}\label{S34}

Let $G=(V,E)$ be a directed hypergraph.
 Define a digraph 
$
\mathfrak N(G)= (V^n_G, E^n_G)
$
where  
$
V^n_G=\{ C\in {\mathbf{P}}(V)\setminus \emptyset | C\in {\mathbf{P}}_{01}(G)\}
$
  and   
$
E^n_G=\{A\to B| (A\to B)\in E\}.
$
 Thus a set  $X\in  {\mathbf{P}}(V)\setminus \emptyset $ is a  vertex of  the digraph  
 $\mathfrak N(G)$ iff  $X$  is an  origin or an end  of an arrow $\mathbf e\in E$.
Any arrow $\bold e=(A\to B)\in E$  gives
an  arrow $(A\to B)\in E^n$.

 \begin{proposition}\label{p3.18} Every morphism of directed hypergraphs $f\colon G\to H$ 
defines a morphism  of digraphs 
$
[\mathfrak N(f)]=(f_{V}^n, f_{E}^n) \colon (V_G^n,E^n_G)\to (V_H^n, E^n_H)
$ 
by
 $
f^n_V(C) \colon =[{\mathbf{P}}(f)]\circ \phi_G(C)
$
and 
$
f_E^n(A\to B) =(f_V(A)\to f_V(B))\in E^n_H
$. Moreover,     $\mathfrak N$ is a functor  from  the category $\mathcal{DH}$ 
to the  category  $\mathcal D$  of  digraphs. \ \ \  $\blacksquare$
\end{proposition}
The composition $\mathfrak D\circ \mathfrak N$  gives a functor from  
 $\mathcal{DH}$  to the category of path complexes.   
For a  directed hypergraph $G$,  we set 
$
H_*^{\bold n}(G)\colon = H_*(\mathfrak D\circ \mathfrak N(G))
$.
These homology groups   will be called the \emph{natural path homology groups}  
of   $G$.  
\begin{example}\label{e3.19} 
\rm  Now we compute  
$H^{\bold{n}}_*(G)$ with  coefficients in $\mathbb R$  of   directed hypergraph  $G=(V_G,E_G)$:
   $$
V_G=\{1,2, 3, 4,5,6,7,8\}, \ \ E_G=\{\bold e_1,\bold  e_2,\bold  e_3, \bold e_4, \bold e_5, \bold e_6,\bold e_7, \bold e_8, \bold e_9\},
$$
 $$
\bold e_1=(\{1\}\to \{3,4\}), \bold e_2=(\{1\}\to \{5, 6\}), \bold e_3=(\{1\}\to \{7,8\}), 
$$
$$
\bold e_4=(\{2\}\to \{3,4\}), \bold e_5=(\{2\}\to \{5,6\}), \bold e_6=(\{2\}\to \{7,8\}), 
$$ 
$$
\bold e_7=(\{3,4\}\to \{5,6\}), \bold e_8=(\{5,6\}\to \{7,8\}), \bold e_9=(\{7,8\}\to \{3,4\}).
$$
The groups $H_*^{\bold n}(G)$ coincide
with the regular path  homology groups of the digraph 
given on  Fig. 3. 

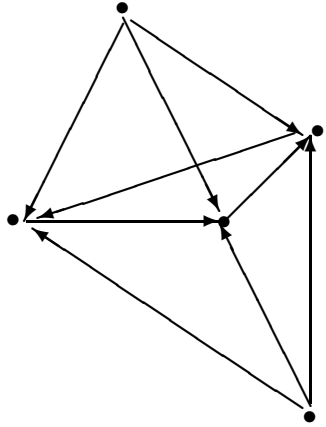
\begin{figure}[h]\label{fig3}
\setlength{\unitlength}{0.14in} \centering
\begin{picture}(16,18)(2,2)
\centering
\thicklines
 \put(5,12){\vector(1,0){7.2}}
\put(4.2,11.8){$\bullet$}
\put(12.1,11.7){$\bullet$}
\put(12.4,12){\vector(1,1){3.2}}
\put(15,15.3){\vector(-3,-1){9.6}}
\put(8.3, 19.7){$\bullet$}
 \put(15.6,15.1){$\bullet$}
\put(8.6, 19.4){\vector(-1,-2){3.7}}
\put(8.6, 19.6){\vector(1,-2){3.6}}
\put(8.9, 19.5){\vector(3,-2){6.4}}
\put(15.3, 4.4){$\bullet$}
\put(15.6, 5.1){\vector(-1,2){3.4}}
\put(15.3, 5){\vector(-3,2){10.1}}
\put(15.6, 5.2){\vector(0,1){10}}
\end{picture}
\caption{The digraph of Example 3.19.}
\end{figure}
\noindent 
Computation gives: 
$
H_0^{\mathbf n}(G)=H_2^{\mathbf n}(G)=\mathbb R$ and 
$H_i^{\mathbf n}(G)=0$ for other $i$.
\end{example}
\begin{lemma}\label{l3.20}  Let $G=(V,E)$ be a directed hypergraph and 
$I_1=(0\to 1)$. 
There is  an equality   
$
\mathfrak N(G)\Box I_1=
\mathfrak{N}(G\Box I_1)
$
of digraphs.
\end{lemma}
\begin{proof}  The digraph 
$\mathfrak N(G)\Box I_1$  has 
$
V_{\mathfrak N(G)\Box I_1}=\left\{C\times \{0, 1\}|\, C\in {\mathbf{P}}_{01}(G)\right\}
$
and  $E_{\mathfrak N(G)\Box I_1}=E_{0,1}\cup E_{0\to 1}$ where 
$
E_{0,1}=\{A\times \{i\}\to B\times\{i\}|
\, A\to B\in E; i=0,1\}$, 
$E_{0\to 1}=\left\{C\times\{0\}\to C\times\{1\}|\, C\in {\mathbf{P}}_{01}(G)\right\}$. 
The digraph 
$\mathfrak N(G\Box I_1)$  has the set of vertices 
$
V_{\mathfrak N(G\Box I_1)}=\left\{C\times \{0, 1\}|\, C\in {\mathbf{P}}_{01}(G)\right\}
$
which  coincides with $V_{\mathfrak N(G)\Box I_1}$
and the set of edges $E_{\mathfrak N(G\Box I_1)}=E_{0,1}\cup E_{0\to 1}$ which  coincides with $E_{\mathfrak N(G)\Box I_1}$. 
\end{proof}

\begin{theorem}\label{t3.21} For a directed hypergraph $G$,  the natural path 
	homology groups $H^{\bold n}_*(G)$ are  homotopy invariant.
\end{theorem}
\begin{proof} The path homology groups defined on the category of digraphs are homotopy 
	invariant \cite{Hyper},\cite{MiHomotopy} and thus,  the result follows from Lemma \ref{l3.20}. 
\end{proof}

\bigskip

Yu.V. Muranov: \emph{Faculty of Mathematics and Computer Science, University
of Warmia and Mazury, Sloneczna 54 Street, 10-710 Olsztyn, Poland }

e-mail: muranov@matman.uwm.edu.pl

\smallskip

Anna Szczepkowska: \emph{
Faculty of Mathematics and 
Computer Science,  University
of Warmia and Mazury, 
10-710 Olsztyn, Poland}

e-mail:  	anna.szczepkowska@matman.uwm.edu.pl
\smallskip

Vladimir Vershinin: \emph{D\'epartement des Sciences Math\'ematiques,
Universit\'e de Montpellier, Place Eug\'ene Bataillon, 34095 Montpellier
cedex 5, France} 

e-mail: vladimir.verchinine@univ-montp2.fr
\end{document}